\documentclass[journal]{IEEEtran}
\usepackage{epsf}
\usepackage{graphicx}
\usepackage{amsfonts,amssymb,amsmath}
\usepackage{algorithmic,algorithm}
\usepackage{subfigure}
\usepackage{latexsym}
\usepackage{bm}
\usepackage[usenames]{color}
\usepackage{url}
\usepackage{amsthm}

\usepackage[normalem]{ulem}
\newcommand{\expect}[1]{\mathbb{E}\left[#1\right]}
\newtheorem{theorem}{Theorem}
\newtheorem{lemma}{Lemma}
\newtheorem{assumption}{Assumption}
\newtheorem{corollary}{Corollary}

\newtheorem{definition}{Definition}

\begin{document}
\title{A Probabilistic Sample Path Convergence Time Analysis of Drift-Plus-Penalty Algorithm for Stochastic Optimization}


\author{Xiaohan Wei, Hao Yu and Michael J. Neely
\thanks{The authors are with the Electrical Engineering department at the University of Southern California, Los Angeles, CA.}

}

\maketitle

\begin{abstract}
This paper considers the problem of minimizing the time average of a controlled stochastic process subject to multiple time average constraints on other related processes. The probability distribution of the random events in the system is unknown to the controller. A typical application is time average power minimization subject to network throughput constraints for different users in a network with time varying channel conditions. We show that with probability at least $1-2\delta$, the classical drift-plus-penalty algorithm provides a sample path $\mathcal{O}(\varepsilon)$ approximation to optimality with a convergence time $\frac{1}{\varepsilon^2}\max\left\{\log^2\frac1\varepsilon\log\frac2\delta,~\log^3\frac2\delta\right\}$, where $\varepsilon>0$ is a parameter related to the algorithm. When there is only one constraint, we further show that the convergence time can be improved to $\frac{1}{\varepsilon^2}\log^2\frac1\delta$.
\end{abstract}




\section{Introduction}
Consider a slotted time system with slots $t\in\{1,2,3,\ldots\}$ with an independent and identically distributed (i.i.d.) process $\{w[t]\}_{t=1}^{\infty}$, which takes values in an arbitrary set $W$ with a probability distribution unknown to the controller.
At each time slot, the controller observes the realization $w[t]$ and picks a decision vector $\mathbf{z}[t]\triangleq(z_0[t],~z_1[t],\ldots,~z_L[t])\in\mathcal{A}(w[t])$, where $\mathcal{A}(w[t])\in\mathbb{R}^{L+1}$ is an option set which possibly depends on $w[t]$.

The goal is to minimize the time average of the objective $z_0[t]$ subject to $L$ time average constraints on processes $z_{l}[t]$, $l=0,1,\ldots,L$, respectively. Let
\begin{align*}
\overline{z}_l=&\limsup_{T\rightarrow\infty}\frac{1}{T}\sum_{t=1}^Tz_l[t],~l\in\{0,1,2,\ldots,L\}.
\end{align*}
Then we write the stochastic optimization problem as
\begin{align}
\min~~&\overline{z}_0\label{obj-problem1}\\
\textrm{s.t.}~~&\overline{z}_l\leq0,~~\forall l\in\{1,2,\ldots,L\},\label{constraint-1-problem1}\\
&\mathbf{z}[t]\in\mathcal{A}(w[t]),~\forall t\in\{1,2,\ldots\},\label{constraint-2-problem1}
\end{align}
where both the minimum and constraints are taken in an almost sure (probability 1) sense. We assume the problem is feasible and the minimum does exist.

\subsection{Related problems and applications}
Problems with the above formulation is common in wireless communications and networking. For example, $w[t]$ can represent a vector of the time varying channel conditions, and $\{z_0[t],\ldots,~z_L[t]\}$ include instantaneous communication rates, power allocations and other metrics for different devices in the network. Specific examples involving this formulation include beamforming (\cite{SDL06}, \cite{NST09}), cognitive radio networks (\cite{QCS08}), energy-aware task scheduling (\cite{Neely2010}, \cite{scheduling-shroff}, \cite{stochastic-scheduling-shroff}) and stock market trading (\cite{Ne10}). In Section \ref{section:simulation}, we present a concrete example of formulation \eqref{obj-problem1}-\eqref{constraint-2-problem1} related to dynamic server scheduling. 

Furthermore, it is shown in Chapter 5 of \cite{Neely2010} that the following more general stochastic convex optimization problem can be mapped back to a problem with structure \eqref{obj-problem1}-\eqref{constraint-2-problem1} via a vector of auxiliary variables:
\begin{align*}
\min~~&f(\overline{\mathbf{z}})\\
\textrm{s.t.}~~&g_k(\overline{\mathbf{z}})\leq0,~~\forall k\in\{1,2,\ldots,K\},\\
&\mathbf{z}[t]\in\mathcal{A}(w[t]),~\forall t\in\{1,2,\ldots\},
\end{align*}
with $f(\cdot)$, $g_k(\cdot)$ continuous and convex and $\mathcal{A}(w[t])$ subset of $\mathbb{R}^M$. The above formulation arises, for example, in network throughput-utility optimization where $\overline{\mathbf{z}}$ represents a vector of achieved throughput, $\mathcal{A}(w[t])$ is a proper subset of $\mathbb{R}^M$ and $f(\cdot)$ is a convex function measuring the network fairness. A typical example of such function when $\mathbf{z}[t]$ is nonnegative is
\[f(\mathbf{z})=-\sum_{m=1}^M\log(1+v_mz_m),\]
where $\{v_m\}_{m=1}^M$ are nonnegative weights.
For more details on network utility optimization, see \cite{network-utility}, \cite{Neely2010}, \cite{atilla-primal-dual-jsac}-\cite{stolyar-greedy}.

\textbf{
Finally, we make several remarks on the i.i.d. assumption of the random event $w[t]$. The i.i.d. assumption is posed here mainly for the ease of analysis, although it appears fairly often in engineering models. For example, in wireless communication scenario, the i.i.d. Rayleigh distribution is often used to model the receiver channel condition. For more detailed elaboration on i.i.d. Rayleigh fading channel, see chapter 7.3.8 of \cite{TV05}.
The i.i.d. assumption is also adopted in a number of other literatures such as \cite{energy-aware}, \cite{ergodic-optimization} and \cite{stochastic-scheduling-shroff}.
}

\subsection{The drift-plus-penalty algorithm}\label{introduction-algorithm}
\textbf{
This subsection briefly introduces the drift-plus-penalty algorithm. This algorithm is previously introduced in \cite{network-utility}, \cite{Neely2010} and \cite{JAM2012} with a provable $\mathcal{O}(\varepsilon)$ approximation solution to \eqref{obj-problem1}-\eqref{constraint-2-problem1} as $T$ goes to infinity. It has been applied to solve various problems in wireless communications and networking (\cite{Neely2010}\cite{GNT06}). For more recent applications in encounter-based network and mobile edge network, see \cite{AWKN16} and \cite{GIT16}. 
}

Define $L$ virtual queues $Q_l[t],~l\in\{1,2,\ldots,L\}$ which are 0 at $t=1$ and updated as follows:
\begin{equation}\label{queue_update}
Q_l[t+1]=\max\left\{Q_l[t]+z_l[t],0\right\}.
\end{equation}
Meanwhile, denote $\mathbf{Q}[t]=(Q_1[t],~Q_2[t],~\ldots~,~Q_L[t])$. The basic intuition behind the virtual queue idea is that if an algorithm can stabilize $Q_l[t],~\forall l\in\{1,2,\ldots,L\}$, then, the average ``rate'' $\overline{z}_l$ is below 0 and the constraints are satisfied.
Then, the algorithm proceeds as follows via a fixed trade off parameter $V>0$:
\begin{itemize}
  \item At the beginning of each time slot $t$, observe $w[t]$, $Q_l[t]$ and take $\mathbf{z}[t]\in\mathcal{A}(w[t])$ so as to solve the following unconstrained optimization problem:
      \begin{equation}\label{dpp_minimization}
      \min_{\mathbf{z}[t]\in\mathcal{A}(w[t])}~~Vz_0[t]+\sum_{l=1}^LQ_l[t]z_l[t].
      \end{equation}
      In other words, the solution treats $Q_l[t]$ and $w[t]$ as given constants and chooses $\mathbf{z}[t]$ in $\mathcal{A}(w[t])$ to minimize the above expression.
  \item Update the virtual queues
      \[Q_l[t+1]=\max\{Q_l[t]+z_l[t],0\}~~\forall l\in\{1,2,\ldots,L\}.\]
\end{itemize}

In the current work, we focus on its
sample path analysis in finite time: It computes the convergence time required to achieve an $\mathcal{O}(\varepsilon)$ approximation with high probability, as discussed in the next subsections.

\subsection{Related algorithms and convergence time}
The algorithm introduced in the last section is closely related to the idea of opportunistic scheduling, which was pioneered by Tassiulas and Ephremides in \cite{queue-stable-tassiulas-1} and \cite{queue-stable-tassiulas-2}. A max-weight algorithm was first introduced in their works to stabilize multiple parallel queues in data networks. The drift-plus-penalty algorithm builds upon max-weight algorithm to further maximize the network utility or minimize energy consumption while stabilizing the queues in the network at the same time (\cite{network-utility} and \cite{Neely2010}). A sample path asymptotic analysis is presented in \cite{JAM2012} using the strong law of large numbers for supermartingale difference sequences. Under mild assumptions on $\mathbf{z}[t]$, it shows that the drift-plus-penalty algorithm satisfies constraints \eqref{constraint-1-problem1}-\eqref{constraint-2-problem1} and achieves the near optimality
\[\overline{z}_0\leq z^{opt}+\mathcal{O}(\varepsilon),\]
with probability 1, where $z^{opt}$ is the minimum achieved by the optimization problem \eqref{obj-problem1}-\eqref{constraint-2-problem1}. \textbf{Throughout the paper, we use the notation $\mathcal{O}(\varepsilon)$ to hide an absolute constant $M$ meaning for all sufficiently small $\varepsilon$, there exists a constant $M>0$ such that $\overline{z}_0\leq z^{opt}+M\varepsilon$.}

Next, consider the problem of convergence time analysis, i.e., the number of slots needed for the desired near optimality to kick in. Most previous works (such as \cite{correlated-scheduling}, \cite{sucha-convergence} and \cite{energy-aware}) focus on
the expected time average performance and only require the constraints to hold in an expected sense. The work in \cite{correlated-scheduling} proves that the same drift-plus-penalty algorithm described in section \ref{introduction-algorithm} gives an $\mathcal{O}(\varepsilon)$ approximation defined in the following manner:
\begin{align}
\frac{1}{T}\sum_{t=1}^T\expect{z_0[t]}&\leq z^{opt}+\mathcal{O}(\varepsilon),\label{expected-near-optimality}\\
\frac{1}{T}\sum_{t=1}^T\expect{z_l[t]}&\leq \mathcal{O}(\varepsilon),~\forall l\in\{1,2,\ldots,L\}, \label{expect-constraint-violation}
\end{align}
with the convergence time $T=\mathcal{O}(1/\varepsilon^2)$. Work in \cite{energy-aware} demonstrates near-optimal $\mathcal{O}(\log(1/\varepsilon)/\varepsilon)$ convergence time for one constraint. An improved convergence time of $\mathcal{O}(1/\varepsilon)$ is shown in \cite{sucha-convergence} in deterministic problems.

The drift-plus-penalty algorithm can also be viewed as a dual algorithm with averaged primals. A similar stochastic dual algorithm for constrained stochastic optimization in \cite{ergodic-optimization} is shown to satisfy the constraints and achieve the near optimality asymptotically with probability 1 as well. A similar $\mathcal{O}(1/\epsilon^2)$ convergence time result for a dual subgradient method is shown in \cite{ozdaglar-dual-subgradient} in the case of deterministic convex optimization. Related work in \cite{scheduling-shroff} applies a dual subgradient method for non-stochastic optimization in a network scheduling problem. The work in \cite{stochastic-scheduling-shroff} further considers the dual subgradient method with stochastic approximations in network scheduling. Other related optimization methods for
queueing networks are also treated
via fluid limits for Markov chains in \cite{atilla-primal-dual-jsac}-\cite{stolyar-gpd-gen}.

\subsection{Contributions and roadmap to proof}
This paper considers the drift-plus-penalty algorithm for stochastic optimization problem \eqref{obj-problem1}-\eqref{constraint-2-problem1} and, for the first time, gives a sample path convergence time result. Specifically, for a general stochastic optimization of the form \eqref{obj-problem1}-\eqref{constraint-2-problem1},
we show that for any $\delta>0$ and $\varepsilon>0$, with probability at least $1-2\delta$, the drift-plus-penalty algorithm gives an $\mathcal{O}(\varepsilon)$ approximation as follows:
\begin{align*}
\frac{1}{T}\sum_{t=1}^Tz_0[t]&\leq z^{opt}+\mathcal{O}(\varepsilon),\\
\frac{1}{T}\sum_{t=1}^Tz_l[t]&\leq \mathcal{O}(\varepsilon),~\forall l\in\{1,2,\ldots,L\},
\end{align*}
with convergence time $T=\frac{1}{\varepsilon^2}\max\left\{\log^2\frac1\varepsilon\log\frac2\delta,~\log^3\frac2\delta\right\}$. Compared to \eqref{expected-near-optimality}-\eqref{expect-constraint-violation}, we removed the expectations at the cost of an extra logarithm factor on the convergence time. Furthermore, when there is only one time average constraint in \eqref{obj-problem1}-\eqref{constraint-2-problem1} (i.e. $L=1$), we show that the convergence time can be improved to $\frac{1}{\varepsilon^2}\log^2\frac1\delta$.

The proof starts by showing the sum-up drift-plus penalty expression is a supermartingale.
The prime difficulty is that the difference sequence of this supermartingale is potentially unbounded, which prevents us from using established concentration inequalities. We overcome this difficulty by truncation. Specifically, in the general case where there are multiple constraints, we
proceed with the following three steps:
\begin{enumerate}
  \item Truncate the original supermartingale using a stopping time, which gives us another supermartingale with bounded difference.
  \item Show that the tail probability of the original supermartingale is upper bounded by the tail probability of the truncated one plus the probability of occurrence of the stopping time.
  \item Bound the tail probability of the truncated supermartingale by a concentration result and bound the probability of stopping time occurrence by an exponential tail bound of the virtual queue processes.
\end{enumerate}
In the special case where there is only one constraint, we show that performing a truncation using a deterministic constant instead of a stopping time is enough to construct a supermartingale with bounded difference, thereby giving a better convergence time result.

\section{Assumptions and Preliminaries}
\subsection{Basic assumptions}\label{section-assumption}
\begin{assumption}\label{assumption-1}
For any $t\in\{1,2,\ldots\}$, the vector $\mathbf{z}[t]\in\mathcal{A}(w[t])$ satisfies
\begin{align*}
&|z_0[t]|\leq z_{\max},\\
&\sqrt{\sum_{l=1}^Lz_l[t]^2}\leq B.
\end{align*}
where $z_{\max}$ and $B$ are positive constants.
\end{assumption}
In addition, we also need the following compactness assumption.
\begin{assumption}
For any $w\in W$, the set $\mathcal{A}(w)$ is a compact subset of $\mathbb{R}^{L+1}$.
\end{assumption}
This assumption is not crucial in our analysis. However, it guarantees that there is always an optimal solution to \eqref{dpp_minimization} within the drift-plus-penalty algorithm, and thereby relieves us from unnecessary complexities in the convergence time analysis.
\begin{assumption}\label{assumption-3}
($\xi$-slackness)
There exists a \textit{randomized stationary policy} $\mathbf{z}^{(\xi)}[t]$ such that all constraints are satisfied with $\xi>0$ slackness, i.e.
\[\expect{z_l^{(\xi)}[t]}\leq-\xi,~~\forall l\in\{1,2,\ldots,L\}.\]
\end{assumption}

The sets $\mathcal{A}(w[t])$ are not required to have any additional structure beyond these assumptions. In particular, the sets $\mathcal{A}(w[t])$ might be finite, infinite, convex, or nonconvex.

\subsection{Interpretation of drift-plus-penalty}\label{interpretation}
We define the squared norm of the virtual queue vector as
\[\|\mathbf{Q}[t]\|^2=\sum_{l=1}^LQ_l[t]^2.\]
Define the drift of the virtual queue vector as follows:
\[\Delta[t]=\frac12\left(\|\mathbf{Q}[t+1]\|^2-\|\mathbf{Q}[t]\|^2\right).\]
The drift-plus-penalty algorithm observes the vector $\mathbf{Q}[t]$ and random event $w[t]$ at every slot $t$, and then makes a decision $\mathbf{z}[t]\in\mathcal{A}(w[t])$ to greedily minimize an upper bound on the \emph{drift-plus-penalty} expression
\[\expect{\Delta[t]+Vz_0[t]~|~\mathbf{Q}[t],w[t]}.\]
To bound $\Delta[t]$, for any $l\in\{1,2,\ldots,L\}$, we square \eqref{queue_update} from both sides and use the fact that $\max\{z,0\}^2\leq z^2$ to obtain:
\[Q_l[t+1]^2\leq Q_l[t]^2+z_l[t]^2+2Q_l[t]z_l[t].\]
According to Assumption \ref{assumption-1},
\begin{equation}\label{pre_dpp_upper_bound}
\Delta[t]
\leq \frac {B^2}{2}+\sum_{l=1}^LQ_l[t]z_l[t].
\end{equation}
Thus, adding $Vz_0[t]$ from both sides and taking the conditional expectation gives,
\begin{align}\label{dpp_upper_bound}
&\expect{\Delta[t]+Vz_0[t]~|~\mathbf{Q}[t],w[t]}\nonumber\\
\leq& \frac {B^2}{2}+\expect{\left.\sum_{l=1}^LQ_l[t]z_l[t]+Vz_0[t]~\right|~\mathbf{Q}[t],w[t]}\nonumber\\
=& \frac {B^2}{2}+\sum_{l=1}^LQ_l[t]z_l[t]+Vz_0[t],
\end{align}
where the equality comes from the fact that given $\mathbf{Q}[t]$ and $w[t]$, the term $\sum_{l=1}^LQ_l[t]z_l[t]+Vz_0[t]$ is a constant. Thus, as we have already seen in section \ref{introduction-algorithm}, the drift-plus-penalty algorithm observes the vector $\mathbf{Q}[t]$ and random event $w[t]$ at slot $t$, and minimizes the right hand side of \eqref{dpp_upper_bound}.

\subsection{Optimization over randomized stationary algorithms}
A key feature of the drift-plus-penalty algorithm is that it is performed without knowing the probability distribution of the random events. Suppose for the moment that the controller does know the probability distribution of the random events. Then, consider the following class of \emph{randomized stationary algorithms}:
At the beginning of each time slot $t$, after observing the random event $w[t]$, the controller selects a decision vector $\mathbf{z}^*[t]\in\mathcal{A}(w[t])$ according to some probability distribution which depends only on $w[t]$.


\textbf{The following lemma shows that the optimal solution to \eqref{obj-problem1}-\eqref{constraint-2-problem1} is achievable over the closure of all one-shot expectations of $\mathbf{z}^*[t]$:
\begin{theorem}[Lemma 4.6 of \cite{Neely2010}]\label{theorem-stat-opt}
Let $z^{opt}$ be the minimum achieved by \eqref{obj-problem1}-\eqref{constraint-2-problem1},
Let $\mathcal{P}$ be the subset of $\mathbb{R}^{L+1}$ consisting of one-shot expectations $\expect{\mathbf{z}^*[t]}$ achieved by all randomized stationary algorithms. Then, there exists a vector $\mathbf{z}^*\in\overline{\mathcal{P}}$, the closure of $\mathcal{P}$, such that
\begin{align}
z_0^*&=z^{opt}\label{iid-obj}\\
z_l^*&\leq0,~\forall l\in\{1,2,\ldots,L\}, \label{iid-constraint}
\end{align}
i.e., the optimality is achievable within $\overline{\mathcal{P}}$.
\end{theorem}
Note that $\overline{\mathcal{P}}$ cannot be explicitly constructed if the controller does not know the probability distribtion of $w[t]$. Thus, Theorem \ref{theorem-stat-opt} cannot be used to compute the optimal solution to \eqref{obj-problem1}-\eqref{constraint-2-problem1}. However, we can use it to prove important results as is shown in the following section.
}

\section{Convergence Time Analysis}
For the rest of the paper, we implicitly assume that all lemmas and theorems are built on Assumption \ref{assumption-1} to \ref{assumption-3} and \eqref{obj-problem1}-\eqref{constraint-2-problem1} is feasible.

\subsection{Construction of a supermartingale}
Define $\mathcal{F}_t$ as the system history up to slot $t$.\footnote{Formally, $\mathcal{F}_t$ is the sigma algebra generated by all random variables from slot 1 to slot $t$, which include $\{w[s]\}_{s=1}^t$, $\{\mathbf{z}[s]\}_{s=1}^t$ and $\{\mathbf{Q}[s]\}_{s=1}^t$.} 
The following lemma illustrates the key feature of the drift-plus-penalty algorithm.
\begin{lemma}
The following inequality holds regarding the drift-plus-penalty algorithm for any $t\in\{1,2,3,\ldots\}$:
\begin{equation}\label{key-feature}
\expect{\left.V(z_0[t]-z^{opt})+\sum_{l=1}^LQ_l[t]z_l[t]~\right|~\mathcal{F}_{t-1}}\leq0,
\end{equation}
\end{lemma}
\begin{proof}
Since the drift-plus-penalty algorithm minimizes the term on the right hand side of \eqref{dpp_upper_bound} over all possible decisions at time $t$, it must achieve a smaller value on that term compared to that of any randomized stationary algorithm $\mathbf{z}^*[t]$, Formally, this is
\[\sum_{l=1}^LQ_l[t]z_l[t]+Vz_0[t]
\leq\sum_{l=1}^LQ[t]z_l^*[t]+Vz_0^*[t].\]
Since for any $l\in\{1,2,\ldots,L\}$, 
$$Q_l[t]=\max\{Q_l[t-1]+z_l[t-1],0\}\in\mathcal{F}_{t-1},$$
taking expectations from both sides regarding $w[t]$ gives
\textbf{
\begin{align*}
&\expect{\left.\sum_{l=1}^LQ_l[t]z_l[t]+Vz_0[t]~\right|~\mathcal{F}_{t-1}}\\
\leq&\expect{\left.\sum_{l=1}^LQ_l[t]z_l^*[t]+Vz_0^*[t]~\right|~\mathcal{F}_{t-1}}\\
=&\sum_{l=1}^LQ_l[t]\expect{z_l^*[t]}+V\expect{z_0^*[t]},
\end{align*}
where the last equality follows from the fact that the randomized stationary algorithm chooses $\mathbf{z}^*[t]$ based only on $w[t]$ and thus independent of the virtual queues at time $t$. Since $\expect{\mathbf{z}^*[t]}\in\overline{\mathcal{P}}$, and above inequality holds for any randomized stationary algorithm, it follows
\begin{align*}
\expect{\left.\sum_{l=1}^LQ_l[t]z_l[t]+Vz_0[t]~\right|~\mathcal{F}_{t-1}}
\leq Vz^*+\sum_{l=1}^LQ_l[t]z^*_l,
\end{align*}
which implies the claim combining with Theorem \ref{theorem-stat-opt}.
}
\end{proof}

With the help of the above lemma, we construct a supermartingale as follows,
\begin{lemma}\label{supMG}
Define a process $\{X[t]\}_{t=0}^\infty$ such that $X[0]=0$ and
\[X[t]=\sum_{i=1}^t\left(V(z_0[i]-z^{opt})+\sum_{l=1}^LQ_l[i]z_l[i]\right).\]
Then, $\{X[t]\}_{t=0}^\infty$ is a supermartingale.
\end{lemma}
\begin{proof}
First, it is obvious that $|X[t]|<\infty$ and $X[t]\in\mathcal{F}_t$ for any $t\geq0$.
Then, by \eqref{key-feature}, the following holds for any $t\geq1$:
\begin{align*}
&\expect{X[t]|\mathcal{F}_{t-1}}\\
=&\expect{\left.V(z_0[t]-z^{opt})+\sum_{l=1}^LQ_l[t]z_l[t]\right|\mathcal{F}_{t-1}}+X[t-1]\\
\leq& X[t-1].
\end{align*}
Thus, $\{X[t]\}_{t=0}^\infty$ is a supermartingale.
\end{proof}

\subsection{Truncation by a stopping time}
Although the process $\{X[t]\}_{t=0}^\infty$ is a supermartingale, its difference sequence $V(z_0[t]-z^{opt})+\sum_{l=1}^LQ_l[t]z_l[t]$ is potentially unbounded because the virtual queue process $\{Q_l[t]\}_{t=1}^\infty$ is not bounded. On the other hand, the bounded difference property is crucial for any well established concentration result to work (see \cite{old&new} for details). The way to circumvent this problem is to use truncation. Intuitively, we can ``stop'' the process whenever the difference gets too large and this stopped process then satisfies the bounded difference property.
The idea of truncation has been used under different scenarios (see \cite{Durrett} for sequential analysis and \cite{Hajek} for queue stability analysis). For basic definitions and lemmas related to stopping time and supermatringale, see Appendix \ref{app:basics}.
For the rest of the paper, define $a\wedge b\triangleq\min\{a,b\}$.

The following lemma introduces a truncated process $\{Y[t]\}_{t=0}^{\infty}$ which has some desired properties.
\begin{lemma}\label{truncated_supMG}
For any $c_1>0$, define
\[\tau\triangleq\inf\{t>0:\|\mathbf{Q}[t]\|>c_1\}.\]
Meanwhile, for any $t\geq0$, define
$$Y[t]=X[t\wedge(\tau-1)].$$
Then, $Y[t]$ has the following two properties:\\
1.The process $\{Y[t]\}_{t=0}^\infty$ is a supermartingale.\\
2.The process $\{Y[t]\}_{t=0}^\infty$ has bounded one-step differences,
        \[\left|Y[t+1]-Y[t]\right|\leq c_2,~\forall t\geq0,\]
where $c_2=2Vz_{\max}+Bc_1$.
\end{lemma}

\begin{proof}
\emph{Proof of Property 1:} In order to apply Theorem \ref{stopping_time} it is enough to show that $\tau-1$ is  a valid stopping time, i.e. $\{\tau-1=t\}\in\mathcal{F}_{t},~\forall t\geq0$. Indeed, since $Q_l[t+1]=\max\{Q_l[t]+z_l[t],0\}\in\mathcal{F}_{t},~\forall t\geq0$ and $\forall l\in\{1,2,\ldots,L\}$, it follows that
      \begin{align*}
      &\{\tau-1=t\}=\{\tau=t+1\}\\
      =&\{\|\mathbf{Q}[t+1]\|>c_1\}\cap\{\|\mathbf{Q}[i]\|\leq c_1,~\forall i\leq t\}\in\mathcal{F}_{t}.
      \end{align*}
Since $\{Y[t]\}_{t=0}^\infty$ is a supermartingale truncated by a stopping time, we apply Lemma \ref{stopping_time} in Appendix \ref{app:basics} to conclude that it is also a supermartingale.

\emph{Proof of Property 2:} The proof is provided in Appendix \ref{app_property_2}.
\end{proof}

The following lemma gives a concentration result for $\{X[t]\}_{t=0}^\infty$ which is a supermartingale with possibly unbounded differences. The result is proved by a union bound argument.
\begin{lemma}\label{supMG_ineq}
For the same sequence $\{X[t]\}_{t=0}^\infty$ defined in Lemma \ref{supMG}, fix a time period $T$ and define the ``bad event'' for each $t\in\{1,2,\ldots,T\}$ as follows:
\[\mathcal{B}_t\triangleq\{\|\mathbf{Q}[t]\|>c_1\}\]
for some $c_1>0$. Then, we have for any $\lambda>0$,
\[Pr\left(X[T]\geq\lambda\right)\leq\exp\left(-\frac{\lambda^2}{2Tc_2^2}\right)
+\sum_{t=1}^TPr\left(\mathcal{B}_t\right).\]
\end{lemma}
\begin{proof}
We first show that for any $t\geq1$, $\left\{X[t]\neq Y[t]\right\}\subseteq \cup_{i=1}^t\mathcal{B}_i$. The proof is simple. Any event $X[t]\neq Y[t]$ is equivalent to $X[t\wedge(\tau-1)]\neq X[t]$. Thus, it follows
$t\wedge(\tau-1)\neq t$. This implies $\tau-1<t$ and $\|\mathbf{Q}[t]\|$ must exceed $c_1$ within the first $t$ slots.
Thus, the event $\left\{X[t]\neq Y[t]\right\}$ must belong to $\cup_{i=1}^t\mathcal{B}_i$.

Now, we can bound the probability that $X[T]$ ever gets large using a union bound, i.e.
\begin{align*}
Pr\left(X[T]\geq\lambda\right)
=&Pr\left(X[T]=Y[T],~Y[T]\geq\lambda\right)\\
&+Pr\left(X[T]\neq Y[T],~X[T]\geq\lambda\right)\\
\leq& Pr\left(Y[T]\geq \lambda\right)+Pr(X[T]\neq Y[T])\\
\leq&\exp\left(-\frac{\lambda^2}{2Tc_2^2}\right)+Pr\left(\cup_{t=1}^T\mathcal{B}_t\right)\\
\leq&\exp\left(-\frac{\lambda^2}{2Tc_2^2}\right)+\sum_{t=1}^TPr(\mathcal{B}_t).
\end{align*}
where the second-to-last inequality uses $\left\{X[t]\neq Y[t]\right\}\subseteq\cup_{i=1}^t\mathcal{B}_i$ and Lemma \ref{azuma-inequality} in Appendix \ref{app:basics}.
\end{proof}

\subsection{Exponential tail bound of the virtual queue}
According to Lemma \ref{supMG_ineq}, it remains to show that the probability that the bad event occurs (i.e. $Pr(\mathcal{B}_t)$) is small.
\textbf{The following preliminary lemma comes from the $\xi$-slackness assumption which states that $\|\mathbf{Q}[t]\|$ is not expected to be very large, which leads to a bound for $Pr(\mathcal{B}_t)$ shown in Cororllary \ref{geo_queue_bound}.}

\textbf{
Note that the key intuition here is the queue length has a exponential tail bound under the well-known max-weight algorithm (See also Lemma 3 in \cite{ES12}). Our drift-plus-penalty algorithm builds upon max-weight algorithm and thus also has exponential tail bound on the queues (Lemma 5 and Corollary 1 below). Thus, intuitively, we are not ``losing too much" if the queue is truncated at some appropriate level. This drives our truncation technique on the constructed supermartingales. 
}

\begin{lemma}\label{geometric_bound}
The following holds for any $t\in\mathbb{N}^+$ under the drift-plus-penalty algorithm,
\begin{align*}
\left|\|\mathbf{Q}[t+1]\|-\|\mathbf{Q}[t]\|\right|&\leq B,\\
\expect{\|\mathbf{Q}[t+1]\|-\|\mathbf{Q}[t]\||\mathbf{Q}[t]}&\leq\left\{
                                \begin{array}{ll}
                                  B, & \hbox{if $\|\mathbf{Q}[t]\|\leq C_0V$;} \\
                                  -\xi/2, & \hbox{if $\|\mathbf{Q}[t]\|> C_0V$.}
                                \end{array}
                              \right.
\end{align*}
where $C_0\triangleq (4z_{\max}+\frac{B^2}{V}-\frac{\xi^2}{4V})/\xi$, $B$, $z_{\max}$ are defined in Assumption \ref{assumption-1} and $\xi$ is defined in Assumption \ref{assumption-3}.
\end{lemma}
\begin{proof}
First of all, by definition of $B$ and $\xi$, we have $B\geq\xi$ and $C_0$ is always positive.
According to Assumption \ref{assumption-1}, the increase (or decrease) of all queues are bounded during each slot, it follows for any $t$,
\begin{align*}
&\left|\|\mathbf{Q}[t+1]\|-\|\mathbf{Q}[t]\|\right|\\
&\leq\|\mathbf{Q}[t+1]-\mathbf{Q}[t]\|\\
&=\sqrt{\sum_{l=1}^L(\max\{Q_l[t]+z_l[t],0\}-Q_l[t])^2}\\
&\leq\sqrt{\sum_{l=1}^Lz_l[t]^2}\leq B,
\end{align*}
where the first inequality follows from triangle inequality and the second from the last inequality follows from $|\max\{a+b,0\}-a|\leq |b|,~\forall a,b\in\mathbb{R}$.

Next, suppose $\|\mathbf{Q}[t]\|> C_0V$. Then, since the drift-plus-penalty algorithm minimizes the term on the right hand side of \eqref{dpp_upper_bound} over all possible decisions at time $t$, it must achieve smaller value on that term compared to that of $\xi$-slackness policy $\mathbf{z}^{(\xi)}[t]$. Formally, this is
\begin{align*}
&\expect{\left.\sum_{l=1}^LQ_l[t]z_l[t]+Vz_0[t]~\right|~\mathbf{Q}[t],w[t]}\\
\leq&\expect{\left.\sum_{l=1}^LQ[t]z_l^{(\xi)}[t]+Vz_0^{(\xi)}[t]~\right|~\mathbf{Q}[t],w[t]}.
\end{align*}
Substitute this bound into the right hand side of \eqref{dpp_upper_bound} and take expectations from both sides to obtain
\begin{align*}
&\expect{\Delta[t]+Vz_0[t]~|~\mathbf{Q}[t]}\\
\leq&\frac{B^2}{2}+\expect{\left.\sum_{l=1}^LQ_l[t]z_l^{(\xi)}[t]+Vz_0^{(\xi)}[t]~\right|~\mathbf{Q}[t]}.
\end{align*}
This implies
\begin{align*}
&\expect{\|\mathbf{Q}[t+1]\|^2-\|\mathbf{Q}[t]\|^2~|~\mathbf{Q}[t]}\\
\leq&B^2+2\sum_{l=1}^LQ_l[t]\expect{\left.z_l^{(\xi)}[t]\right|\mathbf{Q}[t]}+4Vz_{\max}\\
\leq&B^2-2\xi\sum_{l=1}^LQ_l[t]+4Vz_{\max}\\
\leq&B^2-2\xi\|\mathbf{Q}[t]\|+4Vz_{\max},
\end{align*}
where the second inequality follows from the $\xi$-slackness property and the assumption that $z_l^{(\xi)}[t]$ is i.i.d. over slots and hence independent of $Q_l[t]$. This further implies
\begin{align*}
&\expect{\|\mathbf{Q}[t+1]\|^2~|~\mathbf{Q}[t]}\\
\leq&\|\mathbf{Q}[t]\|^2-2\xi\|\mathbf{Q}[t]\|+B^2+4Vz_{\max}\\
=&\|\mathbf{Q}[t]\|^2-2\xi\|\mathbf{Q}[t]\|+B^2+4Vz_{\max}-\frac{\xi^2}{4}+\frac{\xi^2}{4}\\
=&\|\mathbf{Q}[t]\|^2-2\xi\|\mathbf{Q}[t]\|+\frac{B^2+4Vz_{\max}-\frac{\xi^2}{4}}{\xi}\cdot\xi+\frac{\xi^2}{4}\\
=&\|\mathbf{Q}[t]\|^2-2\xi\|\mathbf{Q}[t]\|+C_0V\cdot\xi+\frac{\xi^2}{4}\\
\leq&\|\mathbf{Q}[t]\|^2-\xi\|\mathbf{Q}[t]\|+\frac{\xi^2}{4}=\left(\|\mathbf{Q}[t]\|-\frac\xi2\right)^2,
\end{align*}
where the first inequality follows from the definition of $C_0$ and the second inequality follows from the assumption $\|\mathbf{Q}[t]\|\geq C_0V$. Now take the square root from both sides to obtain
\[\sqrt{\expect{\|\mathbf{Q}[t+1]\|^2~|~\mathbf{Q}[t]}}\leq\|\mathbf{Q}[t]\|-\frac\xi2.\]
By concavity of the $\sqrt{x}$ function, we have $\expect{\left.\|\mathbf{Q}[t+1]\|~\right|~\mathbf{Q}[t]}\leq\sqrt{\expect{\|\mathbf{Q}[t+1]\|^2~|~\mathbf{Q}[t]}}$, thus,
\[\expect{\left.\|\mathbf{Q}[t+1]\|~\right|~\mathbf{Q}[t]}\leq\|\mathbf{Q}[t]\|-\frac\xi2,\]
finishing the proof.
\end{proof}

The following lemma gives us a bound on the moments whenever a random process satisfies the drift condition in Lemma \ref{geometric_bound}. Its proof is given in  \cite{energy-aware}.
\begin{lemma}\label{drift_lemma}
Let $K[n]$ be a real random process over $n\in \{1,2,\ldots\}$ satisfying
\begin{align*}
|K[n+1]-K[n]|&\leq\gamma\\
\expect{K[n+1]-K[n]~|~K[n]}&\leq\left\{
                                 \begin{array}{ll}
                                   \gamma, & \hbox{$K[n]<\sigma$;} \\
                                   -\beta, & \hbox{$K[n]\geq \sigma$.}
                                 \end{array}
                               \right.
\end{align*}
for some positive real-valued $\sigma$, and $0<\beta\leq \gamma$. Suppose $K[0]\in\mathbb{R}$ is finite. Then, at every $n\in\{1,2,\ldots\}$, the following holds:
\[\expect{e^{r K[n]}}\leq D+(e^{r K[1]}-D)\rho^n,\]
where $0<\rho<1$ and
\begin{align*}
r=\frac{\beta}{\gamma^2+\gamma\beta/3},~~
\rho=1-\frac{r\beta}{2},~~
D=\frac{(e^{r\gamma}-\rho)e^{r\sigma}}{1-\rho}.
\end{align*}
\end{lemma}

Using the above two lemmas, we have the following important corollary regarding the virtual queue vector.
\begin{corollary}\label{geo_queue_bound}
The following hold for any $t\in\{1,2,\ldots\}$ under the drift-plus-penalty algorithm,\\
1. Bounded moments of virtual queues:
\begin{equation}
\expect{e^{r\|\mathbf{Q}[t]\|}}\leq D,
\end{equation}
where
$$r=\frac{3\xi}{6B^2+B\xi},~D=\frac{(4e^{rB}+r\xi-4)e^{rC_0V}}{r\xi},$$
$B$ is defined in Assumption \ref{assumption-1}, $\xi$ is defined in \ref{assumption-3} and $C_0$ is defined in Lemma \ref{geometric_bound}. \\
2. Exponential tail bound: For any $c_1>0$,
\begin{equation}\label{high_prob_bound_queue}
Pr\left(\|\mathbf{Q}[t]\|>c_1\right)\leq De^{-rc_1}.
\end{equation}
\textbf{
3. Asymptotic feasibility: For any $l\in\{1,2,\cdots,L\}$
\begin{equation}
\limsup_{T\rightarrow\infty}\frac1T\sum_{t=1}^{T-1}z_l[t]\leq 0,~~w.p.1.
\end{equation}}
\end{corollary}

\begin{proof}
The first part follows directly from Lemma \ref{geometric_bound} and Lemma \ref{drift_lemma} by plugging in $\gamma=B$ and $\beta=\xi/2$ in Lemma \ref{drift_lemma}. The second part follows from
\begin{align*}
Pr\left(\|\mathbf{Q}[t]\|>c_1\right)
&=Pr\left(e^{r\|\mathbf{Q}[t]\|}>e^{rc_1}\right)\\
&\leq \frac{\expect{e^{r\|\mathbf{Q}[t]\|}}}{e^{rc_1}}\leq De^{-rc_1}.
\end{align*}
which is a direct application of Markov inequality. \textbf{For the third part of the claim, taking $c_1=\varepsilon T$ and obtain
\begin{align*}
Pr(Q_l[T]>\varepsilon T)\leq De^{-r\varepsilon T}.
\end{align*}
Thus, we have
\begin{align*}
\sum_{T=1}^{\infty}Pr(Q_l[T]>\varepsilon T)\leq D\sum_{T=1}^{\infty}e^{-r\varepsilon T}<+\infty.
\end{align*}
Thus, by the Borel-Cantelli lemma,
\[Pr\left(Q_l[T]>\varepsilon T~\textrm{for infinitely many}~T\right)=0.\]
Since $\varepsilon>0$ is arbitrary, letting $\varepsilon\rightarrow0$ gives
\[Pr\left(\lim_{T\rightarrow\infty}\frac{Q_l[T]}{T}=0\right)=1.\]
On the other hand, by queue updating rule $Q_l[T]\geq Q_l[1] +\sum_{t=1}^{T-1}z_l[t]=\sum_{t=1}^{T-1}z_l[t]$, and thus, the claim follows.}
\end{proof}

\subsection{Convergence time bound}
The following is our main lemma on the performance of drift-plus-penalty algorithm.
\begin{lemma}\label{main_lemma}
Under the proposed drift-plus-penalty algorithm, for any $\delta\in(0,1)$, and any $T\in\mathbb{N}$,
\begin{align*}
&Pr\left(\frac{1}{T}\sum_{t=1}^T\left(V(z_0[t]-z^{opt})+\sum_{l=1}^LQ_l[t]z_l[t]\right)\right.\\
&\left.\leq CV\frac{\max\left\{\log T\log^{1/2}\frac2\delta,\log^{3/2}\frac2\delta\right\}}{\sqrt{T}}\right)\geq 1-\delta,
\end{align*}
where
$C=2\sqrt{2}(2z_{\max}+\frac{B}{rV}+\frac{B}{rV}\log\left(\frac{8e^{rB}+2r\xi-8}{r\xi}\right)+BC_0)$, $C_0$, $r$ are defined in Lemma \ref{geometric_bound} and Corollary \ref{geo_queue_bound}, respectively, $B$, $z_{\max}$ are defined in Assumption \ref{assumption-1}, and $\xi$ is defined in  Assumption \ref{assumption-3}.
\end{lemma}
\begin{proof}
First of all, according to Corollary \ref{geo_queue_bound} and the definition of $\mathcal{B}_t$,
\begin{align*}
Pr(\mathcal{B}_t)\leq De^{-rc_1}.
\end{align*}
Thus,
\[\sum_{t=1}^TPr(\mathcal{B}_t)\leq DTe^{-rc_1}.\]
Then by Lemma \ref{supMG_ineq}, we have
\begin{equation}\label{interim}
Pr(X[T]\geq\lambda)\leq\exp\left(-\frac{\lambda^2}{2Tc_2^2}\right)+DTe^{-rc_1}.
\end{equation}
For any $\delta\in(0,1)$, set $DTe^{-rc_1}=\delta/2$, so that $c_1=\frac1r\log\frac{2DT}{\delta}$. Then, set
\[\exp\left(-\frac{\lambda^2}{2Tc_2^2}\right)=\delta/2,\]
which, by substituting the definition that $c_2=2Vz_{\max}+Bc_1$, implies
\begin{align*}
\lambda&=\sqrt{2T\log\frac2\delta}\left(2Vz_{\max}+\frac{B}{r}\log\frac{2DT}{\delta}\right)\\
=& \frac{\sqrt{2}B}{r}\sqrt{T}\left(\log T\log^{1/2}\frac2\delta
+\log^{3/2}\frac2\delta\right)\\
&+\sqrt{2}\left(2Vz_{\max}+\frac{B}{r}\log(2D)\right)\sqrt{T}\log^{1/2}\frac 2\delta\\
\leq& \frac{CV}{2}\sqrt{T}\left(\log T\log^{1/2}\frac2\delta
+\log^{3/2}\frac2\delta\right)\\
\leq&CV\sqrt{T}\max\left\{\log T\log^{1/2}\frac2\delta,\log^{3/2}\frac2\delta\right\},
\end{align*}
where the second equality follows from simple algebra, the first inequality follows by substituting the definition of $D$ in Corollary \ref{geo_queue_bound} and doing some simple algebra, and the final inequality follows from the fact that $a+b\leq2\max\{a,b\}$. Substitute this choice of $\lambda$ and the definition of $X[T]$ in Lemma \ref{supMG} into \eqref{interim} gives
\begin{align*}
&Pr\left(\frac{1}{T}\sum_{t=1}^T\left(V(z_0[t]-z^{opt})+\sum_{l=1}^LQ_l[t]z_l[t]\right)\right.\\
&\left.\geq\frac{CV\max\left\{\log T\log^{1/2}\frac2\delta,\log^{3/2}\frac2\delta\right\}}{\sqrt{T}}\right)\leq\delta,
\end{align*}
which implies the claim.
\end{proof}
With the help of Lemma \ref{main_lemma}, we have the following theorem,

\begin{theorem}\label{main-theorem}
Fix $\varepsilon>0$ and $\delta\in(0,1)$ and define $V=1/\varepsilon$. Then, for any $T\geq\frac{1}{\varepsilon^2}\max\left\{\log^2\frac1\varepsilon\log\frac2\delta,~\log^3\frac2\delta\right\}$, with probability at least $1-2\delta$, one has:
\begin{align}
\frac{1}{T}\sum_{t=1}^Tz_0[t]&\leq z^{opt}+\mathcal{O}(\varepsilon),\label{near-optimality}\\
\frac{1}{T}\sum_{t=1}^Tz_l[t]&\leq \mathcal{O}(\varepsilon),~\forall l\in\{1,2,\ldots,L\}, \label{constraint-violation}
\end{align}
and thus the drift-plus-penalty algorithm with parameter $V=1/\varepsilon$ provides an $\mathcal{O}(\varepsilon)$ approximation with a convergence time $\frac{1}{\varepsilon^2}\max\left\{\log^2\frac1\varepsilon\log\frac2\delta,~\log^3\frac2\delta\right\}$.
\end{theorem}

The proof is given in Appendix \ref{proof-theorem-2} by applying Lemma \ref{main_lemma} together with Corollary \ref{geo_queue_bound}. First, using the relation between $\Delta[t]$ and $\sum_{i=1}^LQ_l[t]z_l[t]$ in \eqref{pre_dpp_upper_bound} and the fact that $\sum_{t=1}^T\Delta[t]$ is a telescoping sum equal to $\frac12\|\mathbf{Q}[T+1]\|^2$, we can get rid of the term $\sum_{i=1}^LQ_l[t]z_l[t]$ in Lemma \ref{main_lemma} and prove the $\varepsilon$-suboptimality of the objective. Then, we use the exponential decay of the virtual queue vector in Corollary \ref{geo_queue_bound} to bound the constraint violation.

\section{Improved Convergence Time under One Constraint}
In this section, we show that when the problem \eqref{obj-problem1}-\eqref{constraint-2-problem1} has only one constraint (i.e. $L=1$), we can achieve the same $\varepsilon$ approximation as \eqref{near-optimality} and \eqref{constraint-violation} with probability at least $1-2\delta$ with a convergence time
$\frac{1}{\varepsilon^2}\log^2\frac2\delta$. Compared to the previous result, we have improved by a logarithm factor.
\subsection{A deterministic truncation}
In the previous section, we truncate the original supermartingale $\{X[t]\}_{t=1}^\infty$ using a stopping time. In this section, we show that if there is only one constraint, then, a deterministic truncation is enough to construct a new supermartingale with a bounded difference.

Again assume the $\xi$-slackness assumption holds. The translation of Lemma \ref{geometric_bound} to the case of one constraint implies that:
\begin{align}
\left|Q_1[t+1]-Q_1[t]\right|&\leq B,  \label{one-constraint-1}\\
\expect{Q_1[t+1]-Q_1[t]|Q_1[t]}&\leq\left\{
                                \begin{array}{ll}
                                  B, & \hbox{if $Q_1[t]\leq C_0V$;} \\
                                  -\xi/2, & \hbox{if $Q_1[t]> C_0V$,}
                                \end{array}
                              \right.  \label{one-constraint-2}
\end{align}

\begin{lemma}\label{deter-truncation}
If $L=1$ in problem \eqref{obj-problem1}-\eqref{constraint-2-problem1}, then, the following inequality holds regarding the drift-plus-penalty algorithm for any $t\in\{1,2,3,\ldots\}$ and $V\geq B/C_0$:
\begin{equation}
\expect{\left.V(z_0[t]-z^{opt})+(Q_1[t]\wedge C_0V)z_1[t]~\right|~\mathcal{F}_{t-1}}\leq0,\nonumber
\end{equation}
\end{lemma}
\begin{proof}
Since $Q_1[t]\in\mathcal{F}_{t-1}$, we analyze the conditional expectation for the following two cases:\\
1. If $Q_1[t]\leq C_0V$, then, the key feature inequality \eqref{key-feature} implies that
  \begin{align*}
  &\expect{\left.V(z_0[t]-z^{opt})+(Q_1[t]\wedge C_0V)z_1[t]~\right|~\mathcal{F}_{t-1}}\\
  =&\expect{\left.V(z_0[t]-z^{opt})+Q_1[t]z_1[t]~\right|~\mathcal{F}_{t-1}}\leq0.
  \end{align*}
2. If $Q_1[t]> C_0V$, then,
  \begin{align*}
  -\frac\xi2\geq&\expect{Q_1[t+1]-Q_1[t]|\mathcal{F}_{t-1}}\\
  =&\expect{\max\{Q_1[t]+z_1[t],0\}-Q[t]|\mathcal{F}_{t-1}}\\
  =&\expect{z_1[t]|\mathcal{F}_{t-1}},
  \end{align*}
  where the inequality follows from \eqref{one-constraint-2},
  the first equality follows from the queue updating rule \eqref{queue_update}, and the second equality follows from the fact when $Q_1[t]>C_0V$ and $V\geq B/C_0$, the $\max\{\cdot\}$ can be removed.
  Thus, the following chain of inequalities holds
  \begin{align*}
  &\expect{\left.V(z_0[t]-z^{opt})+(Q_1[t]\wedge C_0V)z_1[t]~\right|~\mathcal{F}_{t-1}}\\
  =&\expect{\left.V(z_0[t]-z^{opt})+C_0Vz_1[t]~\right|~\mathcal{F}_{t-1}}\\
  \leq&\expect{\left.2Vz_{\max}+C_0Vz_1[t]~\right|~\mathcal{F}_{t-1}}\\
  =&2Vz_{\max}+C_0V\expect{z_1[t]|\mathcal{F}_{t-1}}\\
  \leq&2Vz_{\max}-C_0V\xi/2.
  \end{align*}
  Substitute the definition $C_0= (4z_{\max}+\frac{B^2}{V}-\frac{\xi^2}{4V})/\xi$,
  \begin{align*}
  &\expect{\left.V(z_0[t]-z^{opt})+(Q_1[t]\wedge C_0V)z_1[t]~\right|~\mathcal{F}_{t-1}}\\
  \leq&2Vz_{\max}-\frac{V\xi}{2}\cdot\frac{4z_{\max}+\frac{B^2}{V}-\frac{\xi^2}{4V}}{\xi}\\
  =&-(B^2-\xi^2/4)/2<0,
  \end{align*}
since $B$ defined in Assumption \ref{assumption-1} satisfies $|z_1[t]|\leq B,~\forall t$.
\end{proof}
The following corollary follows directly from the above lemma. Its proof is similar to that of Lemma \ref{supMG}.

\begin{corollary}\label{supMG-2}
Define a process $\{G[t]\}_{t=0}^\infty$ such that $G[0]=0$ and
\[G[t]\triangleq\sum_{i=1}^t\left(V(z_0[i]-z^{opt})+(Q_1[i]\wedge C_0V)z_1[i]\right).\]
Then, $\{G[t]\}_{t=0}^\infty$ is a supermartingale.
\end{corollary}

\subsection{A detailed analysis of the truncated queue process}
In Section \ref{interpretation}, we illustrated the relation between $\Delta[t]$ and $\sum_{i=1}^LQ_l[t]z_l[t]$ in \eqref{pre_dpp_upper_bound} thereby using the fact that $\sum_{t=1}^T\Delta[t]$ is a telescoping sum equal to $\frac12\|\mathbf{Q}[T+1]\|^2$ to prove Theorem \ref{main-theorem}. However, since we have truncated the queue in the process $G[t]$, the telescoping relation does not hold for $\sum_{t=1}^T(Q_1[t]\wedge C_0V)z_1[t]$ anymore. The following argument shows that $\sum_{t=1}^T(Q_1[t]\wedge C_0V)z_1[t]$ is actually not far away from a telescoping sum.

Let $n_j$ be the $j$-th time the queue process $Q_1[t]$ visits $[0,C_0V]$ and $n_0=1$.\footnote{If $Q_1[t]$ stays within $[0,C_0V]$ for two consecutive time slots $t$ and $t+1$, then, this is counted as two units.} Define $\tau_j\triangleq n_{j+1}-n_j$ as the inter-visit time period. Let $n_J$ be the last time slot in $\{1,2,\ldots,T+1\}$ that $Q_1[t]$ visits $[0,C_0V]$. The following lemma analyzes the partial sum of $(Q_1[t]\wedge C_0V)z_1[t]$ from 1 to $n_J-1$. Its proof involves a series of algebraic manipulations and is postponed to Appendix \ref{proof-of-telescoping}.

\begin{lemma}
Suppose $n_J>1$, then, the following holds for $V\geq B/C_0$,
\[\left|\sum_{t=1}^{n_J-1}(Q_1[t]\wedge C_0V)z_1[t]-\frac12Q[n_J]^2\right|\leq\frac52B^2(n_J-1),\]
where $B$ is defined in Assumption \ref{assumption-1} and $C_0$ is defined in Lemma \ref{geometric_bound}.
\end{lemma}

The following lemma bounds the time average of the truncated ``drift'' term at any time $T$ (i.e., $\frac1T\sum_{t=1}^T(Q_1[t]\wedge C_0V)z_1[t]$) from below by a constant using the previous lemma,
thereby demonstrating that the ``drift'' term cannot get very small.

\begin{lemma}\label{appr-telescoping-sum}
For any $T\in\mathbb{N}$ and $V\geq B/C_0$, we have
\[\frac1T\sum_{t=1}^T(Q_1[t]\wedge C_0V)z_1[t]\geq-\frac52B^2.\]
\end{lemma}
\begin{proof}
We analyze the following two cases:\\
1. If $n_J=T+1$, then, by the above lemma,
  \begin{align*}
    \sum_{t=1}^T(Q_1[t]\wedge C_0V)z_1[t]=&\sum_{t=1}^{n_J-1}(Q_1[t]\wedge C_0V)z_1[t]\\
    \geq&\frac12Q_1[T+1]^2-\frac52B^2T\\
    \geq&-\frac52B^2T.
  \end{align*}
2. If $n_J<T+1$, then, this implies that $Q_1[t]>C_0V$, for all $t\in\{n_J+1,\ldots,T+1\}$. Thus,
  \begin{align*}
  &\sum_{t=1}^T(Q_1[t]\wedge C_0V)z_1[t]\\
  =&\sum_{t=1}^{n_J-1}(Q_1[t]\wedge C_0V)z_1[t]+\sum_{t=n_J}^{T}(Q_1[t]\wedge C_0V)z_1[t]\\
  \geq&-\frac52B^2(n_{J}-1)+\frac{1}{2}Q_1[n_{J}]^2+\sum_{t=n_J}^{T}(Q_1[t]\wedge C_0V)z_1[t]\\
  =&-\frac52B^2(n_{J}-1)+\frac{1}{2}Q_1[n_{J}]^2+\sum_{t=n_J+1}^{T}C_0Vz_1[t]\\
  &+Q_1[n_J]z_1[n_J]\\
  \geq&-\frac52B^2(n_{J}-1)+\frac{1}{2}Q_1[n_{J}]^2+C_0V\sum_{t=n_J}^{T}z_1[t]-B^2.
  \end{align*}
  where the second to last steps follows from the fact $Q_1[t]>C_0V$ for $t\in\{n_J+1,\ldots,T\}$ and
  the last step follows from $C_0V\geq Q_1[n_J]>C_0V-B$.

  Since $V\geq B/C_0$, it follows, $Q_1[t+1]=Q_1[t]+z_1[t],~\forall t\in\{n_J,\ldots,T\}$.
  Since $Q_1[T+1]>C_0V$ $Q_1[n_J]\leq C_0V$, we have
  \[\sum_{t=n_J}^{T}z_1[t]=Q_1[T+1]-Q_1[n_J]>0.\]
  Thus,
\[\sum_{t=1}^T(Q_1[t]\wedge C_0V)z_1[t] \geq-\frac52B^2n_{J}\geq-\frac52B^2T,\]
finishing the proof.
\end{proof}

\subsection{Convergence time analysis}
The following lemma is a direct application of Lemma \ref{azuma-inequality} and Corollary \ref{supMG-2}.
\begin{lemma}\label{main_lemma-2}
Under the proposed drift-plus-penalty algorithm, for any $\delta\in(0,1)$, any $T\in\mathbb{N}$ and any $V\geq B/C_0$,
\begin{align*}
&Pr\left(\frac{1}{T}\sum_{t=1}^T\left(V(z_0[t]-z^{opt})+(Q_1[t]\wedge C_0V) z_1[t]\right)\right.\\
&\left.\leq 2C_2V\frac{\log(1/\delta)}{\sqrt{T}}\right)\geq 1-\delta,
\end{align*}
where
$C_2=2z_{\max}+C_0B$, $C_0$ is defined in Lemma \ref{geometric_bound}, and $B$, $z_{\max}$ are defined in Assumption \ref{assumption-1}.
\end{lemma}
\begin{proof}
First of all,
we have
\begin{align*}
|G[t]-G[t-1]|=&|V(z_0[t]-z^{opt})+(Q_1[t]\wedge C_0V) z_1[t]|.\\
\leq& 2Vz_{\max}+C_0VB
\end{align*}
Thus, by Lemma \ref{azuma-inequality} and Corollary \ref{supMG-2},
\[Pr(G[T]\geq\lambda)\leq\exp\left(-\frac{\lambda^2}{2TC_2^2V^2}\right).\]
Let $\delta=\exp\left(-\frac{\lambda^2}{2TC_2^2V^2}\right)$, so we get $\lambda=2C_2\sqrt{T}\log\frac1\delta$. Thus,
\[Pr\left(G[T]\geq2C_2V\sqrt{T}\log\frac1\delta\right)\leq\delta,\]
which implies the claim.
\end{proof}

The following is our main theorem regarding the convergence time for this $L=1$ case.
\begin{theorem}\label{main-theorem-2}
Fix $\varepsilon\in(0,C_0/B]$, $\delta\in(0,1)$ and define $V=1/\varepsilon$. Then, for any $T\geq\frac{1}{\varepsilon^2}\log^2\frac1\delta$, with probability at least $1-2\delta$, one has:
\begin{align}
\frac{1}{T}\sum_{t=1}^Tz_0[t]&\leq z^{opt}+\mathcal{O}(\varepsilon),\label{near-optimality-2}\\
\frac{1}{T}\sum_{t=1}^Tz_1[t]&\leq \mathcal{O}(\varepsilon), \label{constraint-violation-2}
\end{align}
and so the drift-plus-penalty algorithm with parameter $V=1/\varepsilon$ provides an $\mathcal{O}(\varepsilon)$ approximation with a convergence time $\frac{1}{\varepsilon^2}\log^2\frac1\delta$.
\end{theorem}

The proof given in Appendix \ref{proof-theorem-3} is similar to the proof of Theorem \ref{main-theorem}. First, we use Lemma \ref{appr-telescoping-sum} to get rid of the term $(Q_1[t]\wedge C_0V) z_1[t]$ in Lemma \ref{main_lemma-2}, thereby proving the $\varepsilon$-suboptimality on the objective. Then, we pass the exponential decay of the virtual queue in Corollary \ref{geo_queue_bound} to one constraint case and bound the constraint violation.

\section{Application on Dynamic Server Scheduling}\label{section:simulation}
\begin{figure}[htbp]
   \centering
   \includegraphics[height=1.5in]{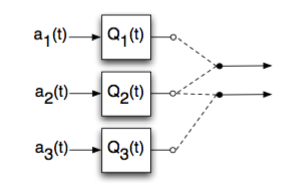} 
   \caption{A 3-queue 2-server system.}
   \label{fig:Stupendous1}
\end{figure}
In this section, we demonstrate the performance of drift-plus-penalty algorithm via a dynamic server scheduling example.
Consider a 3-queue 2-server system with an i.i.d. packet arrival process $\{\mathbf{a}[t]\}_{t=1}^\infty$ shown in Fig. \ref{fig:Stupendous1}. All packets have fixed length and each entry of $\mathbf{a}[t]$ is Bernoulli with mean
\[\expect{\mathbf{a}[t]}=(0.5~0.7~0.4).\]
During each time slot, the controller chooses which two queues to be served by the server. A queue that is allocated a server on a given slot can server exactly one packet on that slot. A single queue cannot receive two servers during the same slot. Thus, the service is given by
 \[
 b_i[t]=
 \begin{cases}
 1,~\textrm{if queue $i$ gets served at time $t$,}\\
 0,~\textrm{otherwise,}
 \end{cases}
 \]
and the service vector $\mathbf{b}[t]\in\{(1,1,0),~(1,0,1),~(0,1,1)\}$. Suppose further that choosing $(1,1,0)$ and $(1,0,1)$ consumes 1 unit of energy whereas choosing $(0,1,1)$ consumes 2 units of energy. Let $p[t]$ be the energy consumption at time slot $t$. 
The goal is to minimize the time average energy consumption while stabilizing all the queues. 
In view of the formulation \eqref{obj-problem1}-\eqref{constraint-2-problem1}, $w[t]=\mathbf{a}[t]$, $z_0[t]=p[t]$, 
$(z_1[t]~z_2[t]~z_3[t])=\mathbf{a}[t]-\mathbf{b}[t]$ and 
$$\mathcal{A}(w[t])=\{\mathbf{a}[t]-(1,1,0),~\mathbf{a}[t]-(1,0,1),~\mathbf{a}[t]-(0,1,1)\}.$$
Thus, we can write \eqref{obj-problem1}-\eqref{constraint-2-problem1} as
\begin{align*}
\min~&\overline{p}\\
s.t.~&\overline{a_i-b_i}\leq0,~i\in\{1,2,3\}.
\end{align*}
where 
$$\overline{p}=\limsup_{T\rightarrow\infty}\frac1T\sum_{t=1}^Tp[t]$$ 
and 
$$\overline{a_i-b_i}=\limsup_{T\rightarrow\infty}\frac1T\sum_{t=1}^T(a_i[t]-b_i[t]).$$
Using drift-plus penalty algorithm, we can solving the problem via the following:
\begin{align*}
\min~~&Vp[t]-\sum_{i=1}^3Q_i[t]b_i[t]\\
s.t.~~&\mathbf{b}[t]\in\{(1,1,0),~(1,0,1),~(0,1,1)\}.
\end{align*}
This can be easily solved via comparing the following values:
\begin{itemize}
\item Option $(1,1,0)$: $V-Q_1[t]-Q_2[t]$.
\item Option $(1,0,1)$: $V-Q_1[t]-Q_3[t]$.
\item Option $(0,1,1)$: $2V-Q_2[t]-Q_3[t]$.
\end{itemize}
Thus, during each time slot, the controller picks the option with the smallest of the above three values, breaking ties arbitrarily. This is a simple dynamic scheduling algorithm which does not need the statistics of the arrivals.

The benchmark we compare to is the optimal stationary algorithm which is i.i.d. over slots. It can be computed offline with the knowledge of the means of arrivals via the following linear program.
\begin{align*}
\min~~&q_1+q_2+2q_3\\
s.t.~~&q_1\geq0,~q_2\geq0,~q_3\geq0,\\
&q_1+q_2+q_3=1,\\
&
\left(
\begin{array}{ccc}
1  & 1  & 0  \\
1  & 0  & 1 \\
0  & 1  & 1  
\end{array}
\right)
\left(
\begin{array}{ccc}
  q_1   \\
  q_2  \\
  q_3   
\end{array}
\right)
\geq
\left(
\begin{array}{ccc}
  0.5   \\
  0.7  \\
  0.4   
\end{array}
\right),
\end{align*}
where $q_1,~q_2,~q_3$ stand for the probabilities of choosing corresponding options. Simple computations gives the solution: $q_1=0.6,~q_2=0.3,~q_3=0.1$ and the average energy consumption is 1.1 unit.

\begin{figure}[htbp]
   \centering
   \includegraphics[height=2.5in]{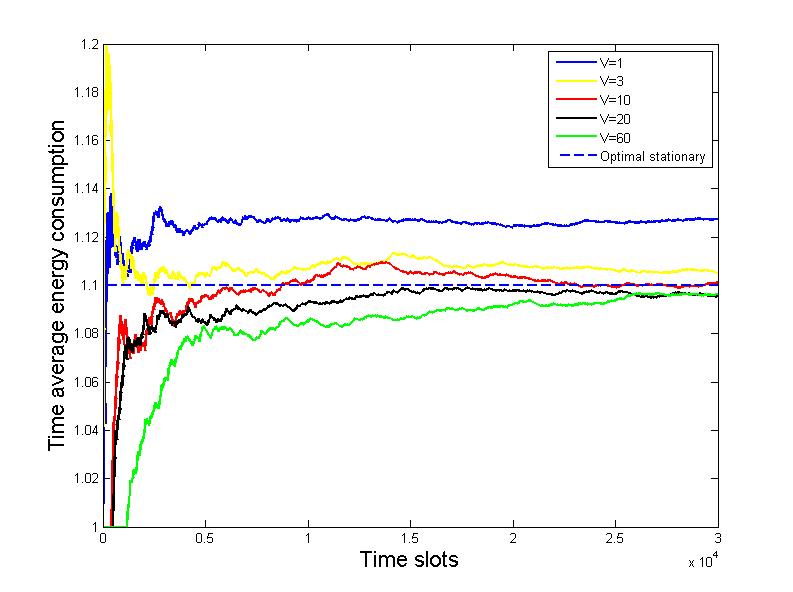} 
   \caption{Time average energy consumption with different V values.}
   \label{fig:Stupendous2}
\end{figure}

Fig. \ref{fig:Stupendous2} plots the time average energy consumption up to time $T$ (i.e., $\frac1T\sum_{t=1}^Tp[t]$) verses $T$. It can be seen that as $V$ gets larger, the time average approaches the optimal average energy consumption but it takes longer to get close to the optimal. Similarly, Fig. \ref{fig:Stupendous3} plots the time average sum-up queue size up to time $T$ (i.e.; $\frac1T\sum_{t=1}^T\sum_{i=1}^3Q_i[t]$) verses $T$. As $V$ gets larger, the time average queue size gets larger and it takes longer to stabilize the queues.

\begin{figure}[htbp]
   \centering
   \includegraphics[height=2.5in]{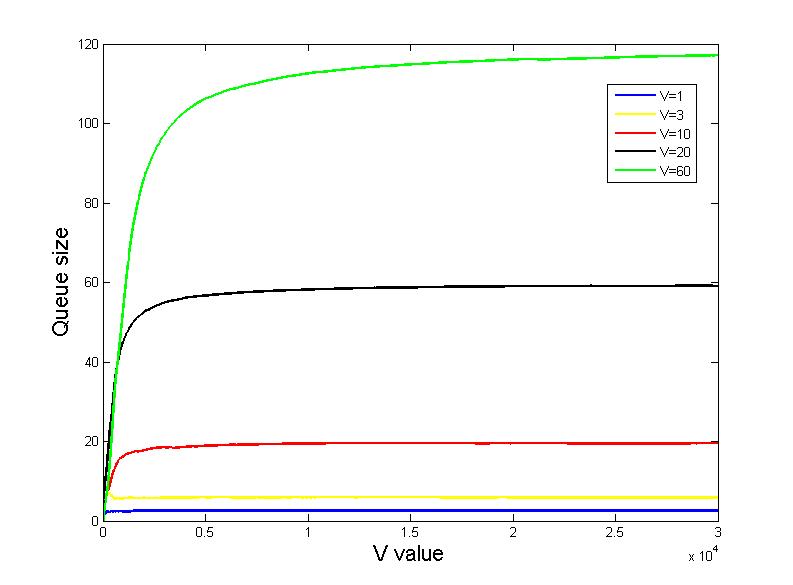} 
   \caption{Time average sum-up queue size with different V values.}
   \label{fig:Stupendous3}
\end{figure}

\section{Conclusions}
This paper analyzes the non-asymptotic performance of the drift-plus-penalty algorithm for stochastic constrained optimization via a truncation technique. With proper truncation level, we show that the drift plus penalty algorithm gives $\mathcal{O}(\varepsilon)$ approximation with a provably bounded convergence time. Furthermore, if there is only one constraint, the convergence time analysis can be improved significantly.

\appendices

\section{Basic definitions and lemmas}\label{app:basics}
The following definition of stopping time can be found in Chapter 4 of \cite{Durrett}.
\begin{definition}
Given a probability space $(\Omega, \mathcal{F}, P)$ and a filtration
$\{\varnothing, \Omega\}=\mathcal{F}_0\subseteq\mathcal{F}_1\subseteq\mathcal{F}_2\ldots$
in $\mathcal{F}$, a stopping time $N$ is a random variable such that for any $n<\infty$,
\[\{N=n\}\in\mathcal{F}_n,\]
i.e. the event that the stopping time occurs at time $n$ is contained in the information up to time $n$.
\end{definition}

The following theorem formalizes the idea of truncation.
\begin{lemma}\label{stopping_time}
(\textit{Theorem 5.2.6 in \cite{Durrett}}) If $N$ is a stopping time and $X[n]$ is a supermartingale, then $X[n\wedge N]$ is also a supermartingale, where $a\wedge b\triangleq\min\{a,b\}$.
\end{lemma}

\begin{lemma}\label{azuma-inequality}
Consider a supermartingale $\{Y[t]\}_{t=0}^\infty$ with bounded difference
\[|Y[t]-Y[t-1]|\leq c_2,~\forall t,\]
and $Y[0]=0$,
For any fixed scale $T$ and any $\lambda>0$, the following concentration inequality holds:
\[Pr(Y[T]\geq\lambda)\leq\exp\left(-\frac{\lambda^2}{2Tc_2^2}\right).\]
\end{lemma}
This is one generalized version of Azuma's inequality to supermartingales. Similar types of supermartingale concentration inequalities and proofs can also be found in Chapter 2 of \cite{old&new}.
\begin{proof}[proof of Lemma \ref{azuma-inequality}]
We first establish a generalized Hoeffding's lemma as follows: For any $i\in\mathbb{N}$ and any $s>0$,
\begin{equation}\label{hoeffding_lemma}
\expect{\left.e^{s(Y[t]-Y[t-1])}\right|\mathcal{F}_{t-1}}\leq e^{s^2c_2^2/2}.
\end{equation}
For simplicity of notations, let $X\triangleq Y[t]-Y[t-1]$, then we have for any $t\in\mathbb{N}$ and any $s>0$,
\begin{align*}
e^{sX}=&e^{sc_2\cdot\frac {X}{c_2}}\leq\frac{1-\frac {X}{c_2}}{2}e^{-sc_2}+\frac{1+\frac {X}{c_2}}{2}e^{sc_2}\\
\leq&\frac12\left(e^{sc_2}+e^{-sc_2}\right)+\frac{X}{2c_2}\left(e^{sc_2}-e^{-sc_2}\right)\\
\leq&e^{s^2c_2^2/2}+\frac{X}{2c_2}\left(e^{sc_2}-e^{-sc_2}\right),
\end{align*}
where the first inequality follows from convexity of the function $e^{sc_2x}$ and the fact that $\frac {X}{c_2}\in[-1,1]$, the third inequality follows from taking Taylor expansion of the term $e^{sc_2}+e^{-sc_2}$. Take conditional expectation from both sides gives
\begin{align*}
\expect{e^{sX}|\mathcal{F}_{t-1}}&=e^{s^2c_2^2/2}+\frac{\expect{X|\mathcal{F}_{t-1}}}{2c_2}\left(e^{sc_2}-e^{-sc_2}\right)\\
&\leq e^{s^2c_2^2/2},
\end{align*}
where the inequality follows from the fact that $Y[t]$ is a supermartingale thus $\expect{X|\mathcal{F}_{t-1}}\leq0$. This finishes the proof of \eqref{hoeffding_lemma}. Now, consider
\begin{align*}
\expect{e^{s Y[T]}}=&\expect{e^{s(Y[T]-Y[T-1]+Y[T-1])}}\\
=&\expect{\expect{\left.e^{s(Y[T]-Y[T-1]+Y[T-1])}\right|\mathcal{F}_{T-1}}}\\
=&\expect{e^{s Y[T-1]}\expect{\left.e^{s(Y[T]-Y[T-1])}\right|\mathcal{F}_{T-1}}}\\
\leq&\expect{e^{s Y[T-1]}}e^{s^2c_2^2/2}.
\end{align*}
By recursively applying above technique $T-1$ times with the fact that $Y[0]=0$, we get
$$\expect{e^{s Y[T]}}\leq e^{\frac{s^2Tc_2^2}{2}}.$$
Finally, by applying above inequality,
\begin{align*}
Pr\left(Y[T]\geq\lambda\right)=&Pr\left(e^{sY[T]}\geq e^{s\lambda}\right)
\leq\frac{\expect{e^{sY[T]}}}{e^{s\lambda}}\\
\leq& e^{-s\lambda+\frac{s^2Tc_2^2}{2}}.
\end{align*}
Optimize the bound regarding $s$ on the right hand side gives the optimal $s^*=\frac{\lambda}{Tc_2^2}$ and the bound follows.
\end{proof}

\section{Proof of Property 2 in Lemma 3}\label{app_property_2}
\begin{proof}
This appendix proves that the truncated process $\{Y[t]\}_{t=0}^\infty$ has bounded difference. We have for any $t\geq0$,
\begin{align*}
  &|Y[t]-Y[t-1]|\\
  =&\left|X[t\wedge(\tau-1)]-X[(t-1)\wedge(\tau-1)]\right|\\
  =&\left|\left(X[t\wedge(\tau-1)]-X[(t-1)\wedge(\tau-1)]\right)1_{\{\tau-1\geq t \}}\right.\\
    &\left.+\left(X[t\wedge(\tau-1)]-X[(t-1)\wedge(\tau-1)]\right)1_{\{\tau-1< t \}}\right|\\
  =&\left|X[t]-X[t-1]\right|1_{\{\tau-1\geq t \}}.
\end{align*}
On the other hand, according to the definition of $X[t]$ in Lemma \ref{supMG}, we have
\begin{align*}
\left|X[t]-X[t-1]\right|&=\left|V(z_0[t]-z^{opt})+\sum_{l=1}^LQ_l[t]z_l[t]\right|\\
&\leq 2Vz_{\max}+\left|\sum_{l=1}^LQ_l[t]z_l[t]\right|\\
&\leq 2Vz_{\max}+B\|\mathbf{Q}[t]\|,
\end{align*}
where the second inequality follows from the Cauchy-Schwarz inequality and Assumption \ref{assumption-1}.
On the set $\{\tau-1\geq t \}$, $\|\mathbf{Q}[t]\|\leq c_1$, thus, we get,
\begin{align*}
\left|X[t]-X[t-1]\right|1_{\{\tau-1\geq t \}}&\leq\left(2Vz_{\max}+Bc_1\right)1_{\{\tau-1\geq t \}}\\
&\leq 2Vz_{\max}+Bc_1,
\end{align*}
finishing the proof.
\end{proof}

\section{Proof of Theorem 2}\label{proof-theorem-2}
\begin{enumerate}
  \item We first show the $\varepsilon$ near optimality. According to \eqref{pre_dpp_upper_bound}, it follows
  \begin{align*}
      &\frac{1}{T}\sum_{t=1}^T\left(V(z_0[t]-z^{opt})+\sum_{l=1}^LQ_l[t]z_l[t]\right)\\
      \geq&\frac{1}{T}\sum_{t=1}^T\left(\Delta[t]-\frac{B^2}{2}+V(z_0[t]-z^{opt})\right)\\
      =&\frac{1}{T}\|\mathbf{Q}[t+1]\|^2-\frac{B^2}{2}+\frac{1}{T}\sum_{t=1}^TV(z_0[t]-z^{opt}).
  \end{align*}
  Thus, by Lemma \ref{main_lemma}, with probability $1-\delta$,
  \begin{align*}
      &\frac{1}{T}\|\mathbf{Q}[t+1]\|^2-\frac{B^2}{2}+\frac{1}{T}\sum_{t=1}^TV(z_0[t]-z^{opt})\\
      \leq& CV\frac{\max\left\{\log T\log^{1/2}\frac2\delta,\log^{3/2}\frac2\delta\right\}}{\sqrt{T}},
  \end{align*}
  which implies
  \begin{align*}
  \frac{1}{T}\sum_{t=1}^Tz_0[t]\leq& z^{opt}+\frac{C\max\left\{\log T\log^{1/2}\frac2\delta,\log^{3/2}\frac2\delta\right\}}{\sqrt{T}}\\
  &+\frac{B^2}{2V}.
  \end{align*}
  Now, substituting $V=1/\varepsilon$ gives
  \begin{align*}
  \frac{1}{T}\sum_{t=1}^Tz_0[t]\leq& z^{opt}+\frac{C\max\left\{\log T\log^{1/2}\frac2\delta,\log^{3/2}\frac2\delta\right\}}{\sqrt{T}}\\
  &+\frac{B^2\varepsilon}{2},
  \end{align*}
  where $C$ has the order $\mathcal{O}(1)$ when $\varepsilon$ is small.
  Then, substitute $T\geq\frac{1}{\varepsilon^2}\max\left\{\log^2\frac1\varepsilon\log\frac2\delta,~\log^3\frac2\delta\right\}$ on the right hand side gives \eqref{single-column}.
  \begin{figure*}
  \normalsize
  \begin{align}
      \frac{C\max\left\{\log T\log^{1/2}\frac2\delta,\log^{3/2}\frac2\delta\right\}}{\sqrt{T}}
      \leq&\frac{C\max\left\{\left(2\log\frac1\varepsilon
      +\log\left(\max\left\{\log\frac1\varepsilon\log^{\frac12}\frac2\delta,~\log^{\frac32}\frac2\delta\right\}\right)\right)\log^{\frac12}\frac2\delta
      ,~\log^{\frac32}\frac2\delta\right\}}
      {\frac1\varepsilon\max\left\{\log\frac1\varepsilon\log^{\frac12}\frac2\delta,~\log^{\frac32}\frac2\delta\right\}}\nonumber\\
      \leq&\frac{C\left(3\log\frac1\varepsilon\log^{\frac12}\frac2\delta+3\log^{\frac32}\frac2\delta\right)}
      {\frac1\varepsilon\max\left\{\log\frac1\varepsilon\log^{\frac12}\frac2\delta,~\log^{\frac32}\frac2\delta\right\}}\label{single-column}
      \leq 6C\varepsilon.
  \end{align}
  \end{figure*}
  Thus, with probability at least $1-\delta$,
  \[\frac{1}{T}\sum_{t=1}^Tz_0[t]\leq z^{opt}+\mathcal{O}(\varepsilon).\]

  \item We then show the constraint violation. By queue updating rule \eqref{queue_update}, for any $t$ and any $l\in\{1,2,\ldots,L\}$, one has
  \[Q_l[t+1]\geq Q_l[t]+z_l[t].\]
  Thus, summing over $t=1,2,\ldots,T$,
  \[Q_l[T+1]\geq Q_l[1]+\sum_{t=1}^Tz_l[t].\]
  Since $Q_l[1]=0$,
  \begin{equation}\label{inter_constraint_violation}
  \frac1T\sum_{t=1}^Tz_l[t]\leq\frac{Q_l[T+1]}{T}.
  \end{equation}
  Recall from Corollary \ref{geo_queue_bound} that we have
  \[Pr\left(\|\mathbf{Q}[T+1]\|>c_1\right)\leq De^{-rc_1}.\]
  Let $De^{-rc_1}=\delta$ and substituting the definition of $D$ in Corollary \ref{geo_queue_bound} give
  $$c_1=\frac1r\log\frac D\delta=\frac1r\log\frac1\delta+C_0V.$$
  Thus,
  \[Pr\left(\|\mathbf{Q}[T+1]\|>\frac1r\log\frac1\delta+C_0V\right)<\delta.\]
  This implies with probability at least $1-\delta$, for any $l\in\{1,2,\ldots,L\}$,
  \[\frac1TQ_l[T+1]\leq\frac{1}{rT}\log\frac1\delta+\frac{C_0V}{T}.\]
  Substitute $T\geq\frac{1}{\varepsilon^2}\max\left\{\log^2\frac1\varepsilon\log\frac2\delta,~\log^3\frac2\delta\right\}$ and $V=1/\varepsilon$ on the right hand side gives
  \[\frac1TQ_l[T+1]\leq\mathcal{O}(\varepsilon).\]
  By \eqref{inter_constraint_violation}, we have with probability at least $1-\delta$,
  \[\frac1T\sum_{t=1}^Tz_l[t]\leq\mathcal{O}(\varepsilon).\]
\end{enumerate}
Since both \eqref{near-optimality} and \eqref{constraint-violation} hold individually with probability at least $1-\delta$, by union bound, with probability at least $1-2\delta$, we have the two conditions hold simultaneously.

\section{Proof of Lemma 10}\label{proof-of-telescoping}
Consider any inter-visit period $n_j$ to $n_{j+1}$,

1. If $\tau_j=1$, then,
\begin{align*}
&\left|\frac{1}{2}\left(Q_1[n_{j+1}]^2-Q_1[n_j]^2\right)-(Q[n_j]\wedge C_0V)z_1[n_j]\right|\\
=&\left|\frac{1}{2}\left(Q_1[n_j+1]^2-Q_1[n_j]^2\right)-Q[n_j]z_1[n_j]\right|\\
\leq& \frac{1}{2}z_1[n_j]^2\leq \frac12B^2,
\end{align*}
where the equality follows from the fact that $Q_1[n_j]\in[0,C_0V]$, and the inequality follows by expanding the $Q_1[n_j+1]^2$ term.

2. If $\tau_j>1$, then,
\begin{align*}
&\left|\frac{1}{2}\left(Q_1[n_{j}+\tau_j]^2-Q_1[n_j]^2\right)-\sum_{t=n_j}^{n_j+\tau_j-1}(Q_1[t]\wedge C_0V)z_1[t]\right|\\
\leq&\left|\frac{1}{2}\left(Q_1[n_{j}+\tau_j]^2-Q_1[n_j+1]^2\right)-C_0V\sum_{t=n_j+1}^{n_j+\tau_j-1}z_1[t]\right|\\
 &+\left|\frac{1}{2}\left(Q_1[n_j+1]^2-Q_1[n_j]^2\right)-Q[n_j]z_1[n_j]\right|\\
 \leq&\left|\frac{1}{2}\left(Q_1[n_{j}+\tau_j]^2-Q_1[n_j+1]^2\right)-C_0V\sum_{t=n_j+1}^{n_j+\tau_j-1}z_1[t]\right|\\
 &+B^2/2\\
 \leq&\left|\frac{1}{2}\left(Q_1[n_{j}+\tau_j]^2-Q_1[n_j+1]^2\right)-Q_1[n_j+1]\right.\\
      &\left.\cdot\sum_{t=n_j+1}^{n_j+\tau_j-1}z_1[t]\right|+\left|(Q_1[n_j+1]-C_0V)\sum_{t=n_j+1}^{n_j+\tau_j-1}z_1[t]\right|\\
      &+B^2/2
\end{align*}
where the first inequality follows from $Q_1[t]\wedge C_0V=C_0V$ from $n_j+1$ to $n_j+\tau_j-1$ and triangle inequality, the second inequality follows from $\tau_j=1$ case and the last inequality follows from triangle inequality again. Now we try to bound the two terms in the last inequality separately.
\begin{itemize}
  \item $\left|\frac{1}{2}\left(Q_1[n_{j}+\tau_j]^2-Q_1[n_j+1]^2\right)-Q_1[n_j+1]\sum_{t=n_j+1}^{n_j+\tau_j-1}\right.$ $\left.z_1[t]\right|\leq 2B^2$:

Since $V\geq B/C_0$ and $Q_1[t]> C_0V$ for any $t\in\{n_j+1, \ldots,n_j+\tau_j-1\}$, according to the queue updating rule, $Q_1[t+1]=Q_1[t]+z_1[t],~\forall t\in\{n_j, \ldots,n_j+\tau_j-1\}$. This gives
\[Q_1[n_j+\tau_j]=Q_1[n_j+1]+\sum_{t=n_j+1}^{n_j+\tau_j-1}z_1[t].\]
Thus,
\begin{align*}
&\left|\frac{1}{2}(Q_1[n_{j}+\tau_j]^2-Q_1[n_j+1]^2)- Q_1[n_j+1]\right.\\
&\left.\cdot\sum_{t=n_j+1}^{n_j+\tau_j-1}z_1[t]\right|=\frac{1}{2}\left|\sum_{t=n_j+1}^{n_j+\tau_j-1}z_1[t]\right|^2.
\end{align*}
Notice that
\begin{align*}
&C_0V\geq Q_1[n_{j}+\tau_j]\geq Q_1[n_{j}+\tau_j-1]-B\geq C_0V-B,\\
&C_0V\leq Q_1[n_{j}+1]\leq Q_1[n_{j}]+B\leq C_0V+B,
\end{align*}
thus,
$\left|\sum_{t=n_j+1}^{n_j+\tau_j-1}z_1[t]\right|^2=\left|Q_1[n_{j}+\tau_j]-Q_1[n_{j}+1]\right|^2\\
\leq (2B)^2=4B^2$,
and this gives the desired bound.

  \item $\left|(Q_1[n_j+1]-C_0V)\sum_{t=n_j+1}^{n_j+\tau_j-1}z_1[i]\right|\leq (\tau_j-1)B^2$:

  Since $C_0V\leq Q[n_j+1]\leq C_0V+B$, it follows
  \[|Q[n_j+1]-C_0V|\leq B.\]
  Thus, the desired bound follows from $|z_1[t]|\leq B$.
\end{itemize}

Above all, we get for $\tau_j>1$,
\begin{align*}
&\left|\frac{1}{2}\left(Q_1[n_{j}+\tau_j]^2-Q_1[n_j]^2\right)-\sum_{t=n_j}^{n_j+\tau_j-1}(Q_1[t]\wedge C_0V)z_1[t]\right|\\
&\leq2B^2+(\tau_j-1)B^2+B^2/2=\left(\tau_j+\frac32\right)B^2.
\end{align*}

Thus, for any $\tau_j$,
\begin{align*}
&\left|\frac{1}{2}\left(Q_1[n_{j}+\tau_j]^2-Q[n_j]^2\right)-\sum_{t=n_j}^{n_j+\tau_j-1}(Q_1[t]\wedge C_0V)z_1[t]\right|\\
&\leq\left(\tau_j+\frac32\right)B^2,
\end{align*}
Take the sums of $j$ from 0 to $J-1$ from both sides gives
\begin{align*}
&\sum_{j=0}^{J-1}\left|\sum_{t=n_j}^{n_j+\tau_j-1}(Q_1[t]\wedge C_0V)z_1[t]
-\frac{1}{2}\left(Q_1[n_{j}+\tau_j]^2\right.\right.\\
&\left.\left.-Q_1[n_j]^2\right)\right|\leq \sum_{j=0}^{J-1}\left(\tau_j+\frac32\right)B^2.
\end{align*}
By triangle inequality,
\begin{align*}
&\left|\sum_{t=0}^{n_J-1}(Q_1[t]\wedge C_0V)z_1[t]-\frac{1}{2}Q_1[n_{J}]^2\right|\\
&\leq\sum_{j=0}^{J-1}\left(\tau_j+\frac32\right)B^2\leq \frac52B^2(n_{J}-1),
\end{align*}
where the last inequality follows from the fact that $\sum_{j=0}^{J-1}\tau_j=n_J-1$ and $J\leq n_J-1$.

\section{Proof of Theorem 3}\label{proof-theorem-3}
First of all, substitute the bound of $\sum_{t=1}^T(Q_1[t]\wedge C_0V) z_1[t]$ derived in Lemma \ref{appr-telescoping-sum} into Lemma \ref{main_lemma-2} gives,
with probability at least $1-\delta$,
\[\frac1T\sum_{t=1}^TV(z_0[t]-z^{opt})\leq2C_2V\frac{\log(1/\delta)}{\sqrt{T}}+\frac52B^2.\]
Divide $V$ from both sides gives
\[\frac1T\sum_{t=1}^T(z_0[t]-z^{opt})\leq2C_2\frac{\log(1/\delta)}{\sqrt{T}}+\frac52\frac{B^2}{V}.\]
Substitute $V=1/\varepsilon$ and $T=\frac{1}{\varepsilon^2}\log^2\frac1\delta$ into the right hand side of the above inequality gives
\begin{equation}\label{inter-near-optimality-2}
\frac1T\sum_{t=1}^T(z_0[t]-z^{opt})\leq2C_2\varepsilon+\frac52B^2\varepsilon=\mathcal{O}(\varepsilon).
\end{equation}

Then, we do the constraint violation.
By queue updating rule \eqref{queue_update}, for any $t$, one has
  \[Q_1[t+1]\geq Q_1[t]+z_1[t].\]
  Thus, summing over $t=1,2,\ldots,T$,
  \[Q_1[T+1]\geq Q_1[1]+\sum_{t=1}^Tz_1[t].\]
  Since $Q_1[1]=0$,
  \begin{equation}\label{inter_constraint_violation-2}
  \frac1T\sum_{t=1}^Tz_1[t]\leq\frac{Q_1[T+1]}{T}.
  \end{equation}
  Recall from Corollary \ref{geo_queue_bound} that we have
  \[Pr\left(Q_1[T+1]>c_1\right)\leq De^{-rc_1}.\]
  Let $De^{-rc_1}=\delta$ and substituting the definition of $D$ in Corollary \ref{geo_queue_bound} give
  $$c_1=\frac1r\log\frac D\delta=\frac1r\log\frac1\delta+C_0V.$$
  Thus,
  \[Pr\left(Q_1[T+1]>\frac1r\log\frac1\delta+C_0V\right)<\delta.\]
  This implies with probability at least $1-\delta$,
  \[\frac1TQ_1[T+1]\leq\frac{1}{rT}\log\frac1\delta+\frac{C_0V}{T}.\]
Substitute $V=1/\varepsilon$ and $T=\frac{1}{\varepsilon^2}\log^2\frac1\delta$ into the right hand side of the above inequality gives
  \[\frac1TQ_1[T+1]\leq\mathcal{O}(\varepsilon).\]
  By \eqref{inter_constraint_violation-2}, we have with probability at least $1-\delta$,
  \[\frac1T\sum_{t=1}^Tz_1[t]\leq\mathcal{O}(\varepsilon).\]

Since both \eqref{near-optimality-2} and \eqref{constraint-violation-2} hold individually with probability at least $1-\delta$, by union bound, with probability at least $1-2\delta$, we have the two conditions hold simultaneously.

\bibliographystyle{unsrt}
\bibliography{bibliography/refs}

\begin{IEEEbiographynophoto}{Xiaohan Wei}
received the B.S. degree in Electrical
Engineering and information science from the
University of Science and Technology of China,
Hefei, China, in 2012, and the M.S. degree with
honor in electrical engineering from the University
of Southern California, Los Angeles, CA, USA, in
2014, and is currently pursuing the Ph.D. degree
in electrical engineering and M.A. degree in applied
mathematics at the University of Southern
California.
His research is in the area of stochastic analysis and
statistical learning theory.
\end{IEEEbiographynophoto}

\begin{IEEEbiographynophoto}{Hao Yu}
received the B.Eng. in Electrical Engineering from Xi'an Jiaotong University, Xi'an, China, in 2008, the M.Phil. degree in Electrical Engineering from the Hong Kong University of Science and Technology, Hong Kong, in 2011. He is currently working towards the Ph.D degree at the Electrical Engineering Department, University of Southern California, Los Angeles. His research interests are in the areas of optimization theory, design and analysis of algorithms, and network optimization.
\end{IEEEbiographynophoto}
\begin{IEEEbiographynophoto}{Michael J. Neely}
received B.S. degrees in both Electrical Engineering and Mathematics from the University of Maryland, College Park, in 1997. He was then awarded a 3-year Department of Defense NDSEG Fellowship for graduate study at the Massachusetts Institute of Technology, where he received an M.S. degree in 1999 and a Ph.D. in 2003, both in Electrical Engineering. He joined the faculty of Electrical Engineering at the University of Southern California in 2004, where he is currently an Associate Professor. His research is in the area of stochastic network optimization. Michael received the NSF Career award in 2008 and the Viterbi School of Engineering Junior Research Award in 2009. He is a member of Tau Beta Pi and Phi Beta Kappa.
\end{IEEEbiographynophoto}

\end{document}